\documentclass[a4paper]{amsart}
\usepackage{graphicx}
\usepackage{amssymb} 
\usepackage{mathtools}
\usepackage{stmaryrd}
\usepackage{mathrsfs}
\usepackage{dsfont}
\usepackage[mathcal]{euscript}
\usepackage{tikz}
\usepackage{tikz-cd}
\tikzset{labl/.style={anchor=south, rotate=270, inner sep=.5mm}}
\usepackage{hyperref}
\hypersetup{%
  bookmarksnumbered=true,%
  colorlinks=true,%
  linkcolor=blue,%
  citecolor=blue,%
  filecolor=blue,%
  menucolor=blue,%
  urlcolor=blue,%
  bookmarksopen=true,%
  bookmarksdepth=2,%
  pageanchor=true}

\title{Ordnung muss sein}

\author{Henning Krause}
\address{Fakult\"at f\"ur Mathematik\\
Universit\"at Bielefeld\\ D-33501 Bielefeld\\ Germany}
\email{hkrause@math.uni-bielefeld.de}

\theoremstyle{plain}
\newtheorem{thm}{Theorem}[section]
\newtheorem{prop}[thm]{Proposition}
\newtheorem{lem}[thm]{Lemma} 

\newtheorem{cor}[thm]{Corollary}

\theoremstyle{definition}
\newtheorem{defn}[thm]{Definition}
\newtheorem{exm}[thm]{Example}

\theoremstyle{remark}
\newtheorem{rem}[thm]{Remark}

\numberwithin{equation}{section}

\newcommand{\card}{\operatorname{card}}

\newcommand{\colim}{\operatorname*{colim}}

\renewcommand{\dim}{\operatorname{dim}}

\newcommand{\End}{\operatorname{End}}

\newcommand{\Ext}{\operatorname{Ext}}

\newcommand{\fl}{\operatorname{fl}}

\newcommand{\Fun}{\operatorname{Fun}}

\newcommand{\Hom}{\operatorname{Hom}}

\newcommand{\Ind}{\operatorname{Ind}}

\renewcommand{\mod}{\operatorname{mod}}

\newcommand{\Ob}{\operatorname{Ob}}

\newcommand{\proj}{\operatorname{proj}} 

\newcommand{\ptfl}{\operatorname{ptfl}}

\newcommand{\rad}{\operatorname{rad}}

\newcommand{\rep}{\operatorname{rep}}

\newcommand{\Ab}{\mathrm{Ab}}

\newcommand{\op}{\mathrm{op}}

\newcommand{\two}{\mathbf{2}}

\newcommand{\iso}{\xrightarrow{\raisebox{-.4ex}[0ex][0ex]{$\scriptstyle{\sim}$}}}

\newcommand{\longiso}{\xrightarrow{\ \raisebox{-.4ex}[0ex][0ex]{$\scriptstyle{\sim}$}\ }}

\newcommand{\lto}{\longrightarrow}

\newcommand{\xto}{\xrightarrow}
\newcommand*{\intref}[2]{\def\tmp{#1}\ifx\tmp\empty\hyperref[#2]{\ref*{#2}}\else\hyperref[#2]{#1~\ref*{#2}}\fi}

\def\A{\mathcal A} 
\def\B{\mathcal B} 
\def\C{\mathcal C}

\def\calS{\mathcal S}

\def\bbC{\mathbb C} 
 
\def\bbK{\mathbb K}

\def\bbQ{\mathbb Q} 
\def\bbR{\mathbb R}

\newcommand{\frf}{\mathfrak{f}}

\newcommand{\frt}{\mathfrak{t}}

\def\a{\alpha}

\def\e{\varepsilon}

\def\p{\phi}
\def\s{\sigma}

\def\Ga{\Gamma}
\def\La{\Lambda}

\def\Om{\Omega}

\begin{document}



\date{\today}

\maketitle

\begin{center}
  \includegraphics[height=40mm]{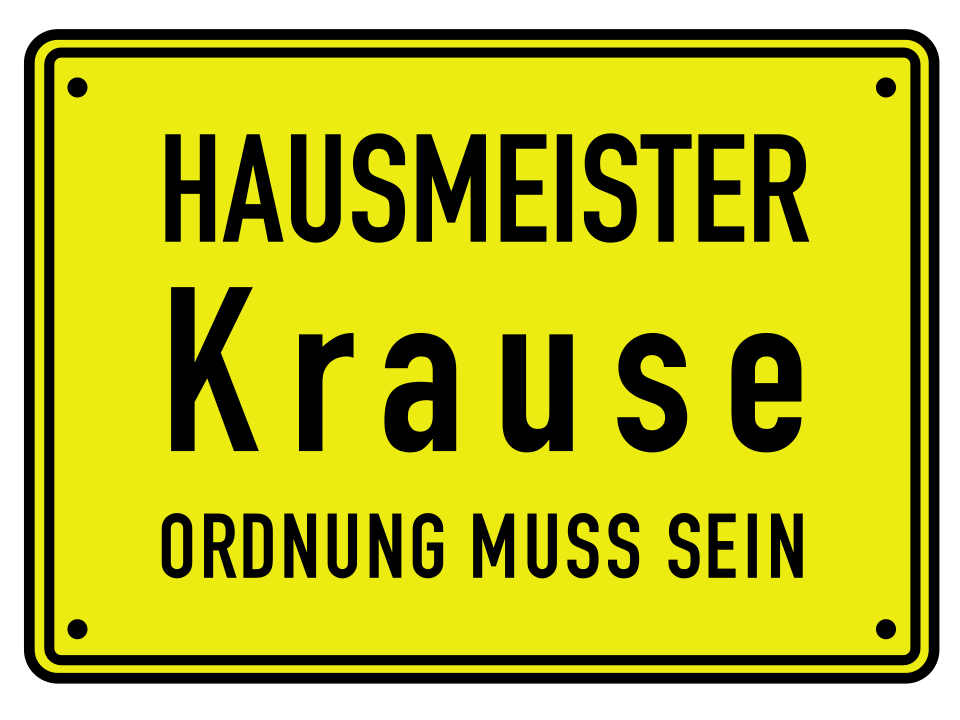}
\end{center}
\bigskip
\begin{abstract}
  For any length category, we establish a set of rules (necessary and
  sufficient) that ensure a partial order on the isomorphism classes
  of simple objects such that the category is equivalent to the
  category of finite dimensional representations of this partially
  ordered set. Equivalently, we characterise the length categories
  that arise as categories of modules over a sheaf of division rings
  on a finite $T_0$-space.
\end{abstract}

\section{Introduction}

A \emph{length category} is an abelian category such that any chain of
subobjects for a given object is finite and the isomorphism classes of
objects form a set \cite{Ga1973}. There are two basic invariants: the
\emph{centre}, which is the ring of natural transformations from the
identity functor to itself, and the \emph{ext-quiver}, for which the
vertices are given by the isomorphism classes of simple objects, and
there is an arrow $[S]\to [T]$ if $\Ext^1(T,S)\neq 0$. If there are no
oriented cycles, then we set $[S]\le [T]$ if there is a path from
$[S]$ to $[T]$ in the ext-quiver and this yields a partial order.
The following result describes a
class of length categories that are completely determined by these
invariants.\footnote{Following the title of the paper, we may refer to
  conditions (L1)--(L3) as `Hausordnung', cf.\
  \url{https://de.wikipedia.org/wiki/Hausmeister_Krause_-_Ordnung_muss_sein}.}

For an abelian category $\A$ we write $Z(\A)$ for its centre and say
that an endomorphism of an object $X$ is \emph{central} if it lies in
the image of the canonical map $Z(\A)\to\End(X)$. The category $\A$ is
\emph{connected} if $0$ and $1$ are the only idempotents in
$Z(\A)$.

\begin{thm}\label{th:main}
  A connected length category is equivalent to the category of
  pointwise finite dimensional representations of a poset over a field
  if and only if the following hold:
  \begin{enumerate}
  \item[(L1)] The ext-quiver  has no oriented cycle.
  \item[(L2)] There exists an object $M$ with a distributive lattice of
    subobjects such that each object is a subquotient of a finite
    direct sum of copies of $M$.
  \item[(L3)] For each simple object all its endomorphisms are central. 
\end{enumerate}
In this case the field equals the centre of the category and the poset
is given by the isomorphism classes of simple objects.
\end{thm}

Here, linear representations of posets are by definition contravariant
functors from the poset (viewed as a category) into the category of
vector spaces.  There is a long tradition of studying such
representations, and we refer to the monograph of Simson for a
detailed account \cite{Si1992}. Intimately related is the study of
incidence algebras and their modules. In fact, the question when a
finite dimensional algebra is an incidence algebra has been
investigated by several authors \cite{BM1979,Fe1977,IK2022}. In
particular, the relevance of distributive modules has been noticed,
but the more general context of length categories seems to be new.

Linear representations of finite posets over a field are
nothing but sheaves on finite $T_0$-spaces with values
in the category of vector spaces; see Remark~\ref{re:sheaves}. This
yields another perspective on the class of length categories arising
in Theorem~\ref{th:main}, which is reminiscent of Gabriel's
reconstruction of a noetherian scheme from its category of
quasi-coherent sheaves \cite[VI]{Ga1962}.

The lattice of subobjects of a finite length object is actually finite
when it is distributive; thus Birkhoff's representation theorem
applies. Roughly speaking, it says that a finite distributive lattice
can be reconstructed from its join-irreducible elements; see
Remark~\ref{re:birkhoff}.  The following analogue of Birkhoff's
theorem is a refinement of Theorem~\ref{th:main}, because we allow the
endomorphism rings of the simple objects to vary.

\begin{thm}\label{th:birkhoff}
  Let $\A$ be a length category and $M\in\A$ a distinguished object satisfying the following:
  \begin{enumerate}
  \item[(M1)] The ext-quiver  has no oriented cycle.
  \item[(M2)] The object $M$ has a distributive lattice of
    subobjects and each object is a subquotient of a finite
    direct sum of copies of $M$.
  \item[(M3)]  For each join-irreducible subobject $N\subseteq M$ the
    canonical map 
    $\End(N)\to\End(N/\rad N)$ is surjective.
\end{enumerate}
Then the set $(M_x)_{x\in\Om}$ of join-irreducible subobjects of $M$
forms a representative set of indecomposable projective objects of
$\A$ and the simple objects are up to isomorphism of the form $S_x=M_x/\rad M_x$.
Moreover, the assignment
\[X\longmapsto \bigoplus_{x\in\Om}\Hom(M_x,X)\] induces an equivalence
$\A\iso\mod A$, where $\mod A$ denotes the category of finitely
presented modules over
$A=\bigoplus_{x,y\in\Om}\Hom(M_x,M_y)$, and
\[M_x\subseteq M_y\quad\iff\quad\Hom(M_x,M_y)\neq 0\quad\iff\quad
  [S_x]\le [S_y].\]
\end{thm}

Let $P$ be a poset and $\bbK=(K_x,\kappa_{xy})_{x,y\in P}$ a family of
divison rings $K_x$ with homomorphisms $\kappa_{xy}\colon K_y\to K_x$
for $x\le y$ such that $\kappa_{xy}\kappa_{yz}=\kappa_{xz}$ for all $x\le y\le
z$. A \emph{$\bbK$-linear representation} of $P$ is given by vector
spaces $M(x)$ over $K_x$ and $K_y$-linear maps
$M(y)\to\kappa_{xy}^*M(x)$ for $x\le y$ satisfying an obvious
compatibility condition.\footnote{A map $M(y)\to\kappa_{xy}^*M(x)$
  corresponds via adjunction to a map $M(y)\otimes_{K_y}K_x\to M(x)$. Then one requires
  for $x\le y\le z$ that $M(z)\otimes_{K_z}K_x\to M(x)$ equals the
  composite of $M(z)\otimes_{K_z}K_y\to M(y)$ tensored with the
  $K_y$-module $K_x$ and $M(y)\otimes_{K_y}K_x\to M(x)$.} One
may think of $\bbK$ as a sheaf of divison rings and then a
$\bbK$-linear representation is nothing but a sheaf of $\bbK$-modules.

\begin{cor}\label{co:birkhoff}
  A length category is equivalent to the category of pointwise finite
  dimensional representations of a poset together with a family of
  division rings if and only if there is a distinguished object $M$
  satisfying the conditions \emph{(M1)--(M3)}.

  In this case the poset is given by the set $(M_x)_{x\in\Om}$ of
  join-irreducible subobjects of $M$, and $x\le y$ if there is an
  inclusion $i_{xy}\colon M_x\to M_y$. For $x\in\Om$ the divsion ring
  is $K_x=\End(M_x)$ and $\kappa_{xy}\colon K_y\to K_x$ is given by
  $i_{xy}\kappa_{xy}(\a)=\a i_{xy}$ for $\a\in K_y$.
\end{cor}  

There is a generalisation of Theorem~\ref{th:main} which covers
infinite posets. Linear representations of infinite posets arise for
example in persistence theory \cite{Ou2015}. In order to deal with
abelian categories having infinitely many isomorphism classes of
simple objects we introduce the concept of a \emph{pointwise length
  category} which agrees with that of a length category when there are
only finitely many simples. For instance, the category of pointwise
finite dimensional representations of a poset is a pointwise length
category provided the poset is down-finite. We refer to the last
section of this note for details and to Theorem~\ref{th:main2}
for a precise statement that generalises Theorem~\ref{th:main}.

\section{Proofs}

In this section we provide the proofs of our main theorems and this
requires several preparations. Throughout we fix a length category
$\A$. We choose a representative set of simple objects
$(S_x)_{x\in \Omega}$ and set
\[\Ga_x\coloneqq \End(S_x).\]
When the ext-quiver of $\A$ is acyclic then $\Omega$ is partially
ordered via $x\le y$ if there is a path from $S_x$ to $S_y$ in the
ext-quiver of $\A$.

\subsection*{Distributive objects}

For any object $M$ and a simple object
$S$ let $[M:S]$ denote the multiplicity of $S$ in a composition series
of $M$. The object $M$ is \emph{multiplicity free} if  $[M:S]\le 1$
for every simple object $S$.

Recall that an object has a distributive lattice of subobjects if and
only if there is no simple object $S$ such that $S\oplus S$ arises as
a subquotient; see for instance the first exercise in
\cite{Be1991}. Objects with this property are called
\emph{distributive}. An object is \emph{local} if it has a unique
maximal subobject.

\begin{lem}\label{le:cycle}
  Let $M$ be a local object and set $T=M/\rad M$. If $S$ is a
  composition factor of $\rad ^n M$ for some $n\ge 0$ , then there is
  a path of length at least $n$ from $[S]$ to $[T]$ in the ext-quiver.
\end{lem}
\begin{proof}
  Set $M^i\coloneqq \rad^i M$ for $i\ge 0$. Then $M^r=0$ for some $r\ge 0$,
  and this yields a filtration
  \[0=M^r\subseteq\cdots \subseteq M^1 \subseteq M^0 =M\] with
  semisimple factors.  The assumption means there is an index $i\ge n$ such that
  $S\subseteq M^{i}/M^{i+1}$. When $i> 0$ 
  there exists a simple object
  $S_1\subseteq M^{i-1}/M^{i}$ and an arrow $[S]\to [S_1]$ in the ext-quiver,
  which comes from the extension
  \[0\lto M^{i}/M^{i+1}\lto M^{i-1}/M^{i+1}\lto M^{i-1}/M^{i}\lto 0.\]
  Inductively this yields a path from $[S]$ to $[T]$ of length $i$.
\end{proof}

\begin{lem}\label{le:mfree}
Every multiplicity free object is distributive. The converse holds
when the ext-quiver has no oriented cycles.
\end{lem}
\begin{proof}
  The first assertion is clear, given the characterisation of
  distributive objects. Now fix an object $M$ and a simple object
  $S$. If $S$ is a composition factor of $M$, then there exists a
  subobject $M'\subseteq M$ with an epimorphism $M'\to  S$. We may
  choose $M'$ of minimal length and then $M'$ is local.
  Now suppose that
  $[M:S]>1$. Then there are two subobjects $M',M''$ with that
  property. If $M''\subseteq M'$, then one obtains a cycle in the
  ext-quiver by Lemma~\ref{le:cycle} since $S$ is a composition factor
  of $\rad M'$. If $M'$ and $M''$ are not
  comparable, then $S\oplus S$ is a subobject of $M/(\rad M'+\rad
  M'')$. Thus $M$ is not distributive.
\end{proof}

\subsection*{Projective covers}

Let $S$ be a simple object and set $\Ga:=\End(S)$. For an object $X$ in
$\A$ let $\langle X\rangle$ denote the full subcategory of $\A$
consisting of all subquotients of finite direct sums of copies of $X$.

\begin{lem}\label{le:Hom-subquot}
  Let $M\to S$ be an essential epimorphism and suppose that the
  canonical map $\End(M)\to\Ga$ is an isomorphism.  For each $X\in\A$
  we have \begin{equation*}\label{eq:Hom-subquot} \dim_\Ga\Hom(M,X)\le
    [X:S].
\end{equation*}    
If equality holds, then it also holds for all objects in $\langle X\rangle$
and $\Hom(M,-)$ is exact on $\langle X\rangle$.
\end{lem}
\begin{proof}
  The induced map $\Hom(S,X)\to\Hom(M,X)$ is an isomorphism when $X$
  is simple, since $M\to S$ essential. This yields the inequality for
  $\ell(X)=1$, and the general case follows by induction on $\ell(X)$,
  using that each exact sequence $0\to X'\to X\to X''\to 0$ in $\A$
  induces an exact sequence
  \[0\to\Hom(M,X')\to\Hom(M,X)\to\Hom(M,X'').\] If there is equality
  for $X$, then also for all objects in $\langle X\rangle$, and this
  yields the exactness of $\Hom(X,-)$.
\end{proof}

\begin{lem}\label{le:End}
  Suppose the ext-quiver is acyclic and let $M$ be a local object.
  Then $\End(M)$ is a divison ring.
\end{lem}
\begin{proof}
  Let $\p\in\End(M)$ be non-invertible. The image of $\p$ is a
  submodule of $\rad M$ and we may apply Lemma~\ref{le:cycle}. Thus
  the radical filtration of $M$ yields a cycle $[S]\to \cdots \to [S]$
  for $S=M/\rad M$ in the ext-quiver when $\p\neq 0$.
\end{proof}

\begin{lem}\label{le:max}
  Suppose the ext-quiver of $\A$ is acyclic. Let $M\to S$ be an essential
  epimorphism such that the canonical map $\End(M)\to\Ga$ is surjective,
 $\langle M\rangle=\A$, and $[M:S]=1$. Then $M$ is projective.
\end{lem}
\begin{proof}
We have $\End(M)\cong \Ga$ by Lemma~\ref{le:End}, and
  then Lemma~\ref{le:Hom-subquot} shows that $M$ is projective.
\end{proof}

\subsection*{Torsion pairs}

For a subset $I\subseteq \Omega$ let $\A_I$ denote the Serre
subcategory of $\A$ consisting of all objects in $\A$ having their
composition factors in $\{S_x\mid x\in I\}$. The inclusion $\A_I\to\A$
admits a right adjoint $\frt_I\colon \A\to\A_I$ taking an object to
its maximal subobject in $\A_I$ and a left adjoint
$\frf_I\colon \A\to\A_I$ taking an object to its maximal quotient in
$\A_I$.

\begin{lem}\label{le:torsion}
  Let $\Omega=\Omega'\sqcup \Omega''$ be a decomposition such that
  $\Ext^1(S_{x'}, S_{x''})=0$ for all $x'\in \Omega',x''\in
  \Omega''$. Then $(\A_{\Omega'},\A_{\Omega''})$ is a torsion pair for
  $\A$ and each object $X\in\A$ fits into an exact sequence
  \begin{equation*}\label{eq:tor}
    0\lto \frt_{\Omega'}(X)\lto X\lto \frf_{\Omega''}(X)\lto 0.
\end{equation*}    
Moreover, the functors $\frt_{\Omega'}$ and  $\frf_{\Omega''}$ are exact.
\end{lem}
\begin{proof}
    The category $\A_{P'}$ is a torsion class for $\A$ and we claim that
  $\A_{P''}$ equals the corresponding torsionfree class. All objects
  from $\A_{P''}$ are torsionfree. So suppose there is a torsionfree
  object $X$ that is not in $\A_{P''}$. We may assume it is of minimal
  length and let $X'\subseteq X$ be the maximal subobject in
  $\A_{P''}$. The ext-vanishing condition implies that $X/X'$ is
  torsionfree. Then $X/X'$ belongs to $\A_{P''}$, and therefore $X$
  belongs to $\A_{P''}$ since $\A_{P''}$ is closed under
  extension. This is a contradiction. Thus $(\A_{P'},\A_{P''})$ is a
  torsion pair.

The functor $\frt_{P'}$ is left exact because it is a right
adjoint. An application of the snake lemma yields its rightexactness,
using that $\Hom(X'',X')=0$ for all $X'\in \A_{P'}$ and
$X''\in \A_{P''}$. The exactness of $\frf_{P''}$ is dual.
\end{proof}

\subsection*{The centre}

Recall that the centre $Z(\A)$ of $\A$ is the ring of natural
transformations from the identity functor to itself.

\begin{lem}\label{le:centre}
  Let $\A$ be a length category such that the ext-quiver of $\A$ is
  acyclic and connected. Then the following are equivalent.
  \begin{enumerate}
  \item For each simple object all its endomorphisms are central.
  \item The centre of $\A$ is a field and isomorphic to the
    endomorphism ring of each simple object.
  \end{enumerate}
\end{lem}
\begin{proof}
  One direction is clear. Thus assume that (1) holds. For any index
  $x$ let $I_x$ denote the kernel of the canonical homomorphism
  $Z(\A)\to \Ga_x$.  Then $Z(\A)/I_x\iso \Ga_{x}$ since we assume that
  all elements of $\Ga_{x}$ are central.  For each pair $x,y$ the
  action of $Z(\A)$ on $\Ext^1(S_x,S_y)$ factors through $\Ga_x$ and
  through $\Ga_y$.  Thus $I_x=I_y$ when $\Ext^1(S_x,S_y)\neq 0$.
  Because the ext-quiver is connected, it follows that $I_x=I_y$ for
  all $x,y$. It remains to show that $I_x=0$ for all $x$. Choose
  $\eta\in I_x$ and suppose $\eta_M\neq 0$ for an object $M$. We may
  choose $M$ of minimal length and decompose the simples arising as
  composition factors into proper subsets $\calS_0$ and $\calS_1$ such
  that $\Ext^1(S_0,S_1)=0$ for all $S_i\in\calS_i$. This is possible
  since the ext-quiver of $\A$ is acyclic and yields a short exact
  sequence $0\to M_0\to M\to M_1\to 0$ such that all composition
  factors of $M_i$ are in $\calS_i$. Then $\eta_M$ induces a
  non-zero morphism $M_1\to M_0$ since $\eta_{M_0}=0=\eta_{M_1}$. This is a
  contradiction and therefore $\eta=0$.
\end{proof}

Given the equivalent conditions of the lemma, an induction on length
shows that $\Hom(X,Y)$ is a module of finite length over $Z(\A)$ for
each pair of objects $X,Y$.

\subsection*{Representations of posets}

Fix a poset $P$ and a field $\bbK$. It is convenient to view $P$ as a
category: the objects are the elements of $P$ and there is a unique
morphism $x\to y$ if and only if $x\le y$. For a ring $A$ let $\mod A$
denote the category of finitely presented right $A$-modules.
A \emph{pointwise
  finite dimensional representation} of $P$ is by definition a functor
$P^\op\to\mod\bbK$, and we set
\[\rep(P,\bbK)\coloneqq \Fun(P^\op,\mod\bbK).\]

Let $\bbK P$ denote the \emph{incidence algebra} which is a
$\bbK$-linear space with basis given by elements $a_{xy}$ for any pair
$x\le y$ in $P$ and multiplication induced by
\begin{equation*}\label{eq:mult-a}
  a_{xy}a_{vw}:=\delta_{wx} a_{vy}. 
\end{equation*}
For $M$ in $\rep(P,\bbK)$ there is a canonical right
action of $\bbK P$ on $\bigoplus_{x\in P} M(x)$.

\begin{lem}\label{le:rep}
  Let $P$ be finite. Then the assignment
  $M\mapsto \bigoplus_{x\in P} M(x)$ induces an equivalence
  \[\rep(P,\bbK)\iso\mod \bbK P.\]
\end{lem}
\begin{proof}
Taking a $\bbK P$-module $N$ to the representation of $P$ given by \[x\mapsto
N(x):=Na_{xx}\] and for $x\le y$ the linear map $N(y)\to N(x)$ given by
right multiplication with $a_{xy}$ yields a quasi-inverse.
\end{proof}

\subsection*{Proofs of the main theorems}

We are ready to prove the theorems from the introduction and keep our
notations for a length category $\A$.

\begin {proof}[Proof of Theorem~\ref{th:birkhoff}]
  Suppose that conditions (M1)--(M3) hold for $\A$ and $M\in\A$. Let
  $(M_x)_{x\in\Psi}$ be the set of join-irreducible subobjects of
  $M$. Then we claim the following:
  \begin{enumerate}
  \item For $x\in \Om$ and $\downarrow\! x=\{y\in\Om\mid y\le x\}$
    there is $x'\in\Psi$ such that
  \[N_x\coloneqq \frt_{\downarrow x}(M)= M_{x'}.\]
    \item The assignment $x\mapsto x'$ gives a bijection $\Om\to\Psi$. 
\item  Each  $M_x$ is a projective object.
\end{enumerate}
We apply Lemma~\ref{le:torsion} using the decomposition
$\Om=\downarrow\!  x\sqcup(\Om\smallsetminus\downarrow\! x)$.  Thus
$N_x$ is distributive and all objects of $\A_{\downarrow x}$ are
subquotients of finite direct sums of $N_x$ since the functor
$\frt_{\downarrow x}$ is exact. Moreover, $\frf_{\{x\}}(N_x)\cong S_x$
and therefore $N_x$ is join-irreducible in the lattice of subobjects
of $M$. It follows from Lemma~\ref{le:max} that $N_x$ is projective in
$\A_{\downarrow x}$, and then also projective in $\A$ since the
inclusion $\A_{\downarrow x}\to\A$ is left adjoint to an exact
functor. Also, $N_x=M_{x'}$ for some $x'\in\Psi$. On the other hand,
if $N\subseteq M$ is join-irreducible, then $N/\rad N\cong S_x$ for
some $x\in\Om$. All composition factors of $N$ are of the form $S_y$
for some $y\le x$ by Lemma~\ref{le:cycle}, and therefore $N=N_x$.

The object $P\coloneqq\bigoplus_{x\in\Om}M_x$ is a projective
generator for $\A$ and then $\Hom(P,-)$ induces an equivalence
$\A\iso\mod A$ for $A\coloneqq\End(P)$.

For the final assertion we identify $\Om=\Psi$. The implication
\[\Hom(M_x,M_y)\neq
  0\quad\implies\quad [S_x]\le [S_y]\] follows from Lemma~\ref{le:cycle}. If
$[S_x]\le [S_y]$ then \[M_x=\frt_{\downarrow x}(M)\subseteq \frt_{\downarrow y}(M)=M_y.\qedhere\]
\end{proof}

Recall that the \emph{Hasse diagram} of a poset $P$ is the quiver
with vertex set $P$ and an arrow $x\to y$ if $x<y$ and there is no $z$
with $x<z<y$.

\begin {proof}[Proof of Theorem~\ref{th:main}]
  Let $\A$ be a length category. Suppose first that $\A$ is equivalent
  to $\rep(P,\bbK)$ for some poset $P$ and a field $\bbK$. The
  simple representations identify with the elements of
  $P$. Specifically, for $x\in P$ the simple representation $S_x$ is
  given by $S_x(y)=\delta_{xy}\bbK$ for $y\in P$. In particular, it
  satisfies $\End(S_x)\cong \bbK$.  The ext-quiver identifies with the
  Hasse diagram of $P$. Thus there are no oriented cycles. Now
  consider the representation $M$ given by $M(x)=\bbK$ and
  $M(y)\to M(x)$ the identity map for all $x\le y$. Then $[M:S_x]=1$
  for all $x$, and therefore the lattice of subobjects is
  distributive. The subrepresentation $M_x\subseteq M$ satisfying
  $M_x(y)\neq 0$ if and only if $y\le x$ yields a projective cover of
  $S_x$. Thus every indecomposable projective representation embeds
  into $M$, and each object is a subquotient of $M^r$ for some
  $r\ge 1$.

  Now suppose that conditions (L1)--(L3) hold for $\A$. First observe
  that
\[\bbK\coloneqq Z(\A)\] is a field by
Lemma~\ref{le:centre}. Let $M$ be a distributive object such that each
object is a subquotient of a finite direct sum of copies of $M$. Note
that condition (L3) implies (M3). Thus it follows from Theorem~\ref{th:birkhoff}
that the collection $(M_x)_{x\in\Om}$ of join-irreducible subobjects
of $M$ provides projective covers $M_x\to S_x$ of the simple objects
in $\A$. We have \[\dim_\bbK\Hom(M_x,M_y)=[M_y:S_x]\] and therefore
$\Hom(M_x,M_y) \cong \bbK$ if and only if $x\le y$, with a basis
element given by the inclusion $i_{xy}\colon M_x\to M_y$. The
assignment $a_{xy}\mapsto i_{xy}$ induces a $\bbK$-algebra isomorphism
$\bbK\Om\iso A$ for $A=\bigoplus_{x,y\in\Om}\Hom(M_x,M_y)$ and
therefore an equivalence
\[\A\iso\mod A\iso\mod\bbK\Om\iso\rep (\Om,\bbK)\]
if we compose with the equivalences from Theorem~\ref{th:birkhoff} and
Lemma~\ref{le:rep}. More explicitly, the equivalence takes $X\in\A$ to
the representation of $\Om$ given by $x\mapsto\Hom(M_x,X)$.
\end{proof}

\begin {proof}[Proof of Corollary~\ref{co:birkhoff}]
  The statement of Corollary~\ref{co:birkhoff} generalises Theorem~\ref{th:main} and
  the proof is along the same lines, using the full generality of
  Theorem~\ref{th:birkhoff}. We omit the details but explain where the
  maps $\kappa_{xy}\colon K_y\to K_x$ come from when $x\le y$ in
  $\Om$. In fact, $\Hom(M_x,M_y)$ is a $K_y$-$K_x$-bimodule which
  is of length one over $K_x$. This yields by definition a ring
  homomorphism $K_y\to K_x$.
\end{proof}

\section{Remarks}

In this section we provide several comments on the main theorems.

\begin{rem}
In Theorem~\ref{th:main} there is no assumption on the poset $P$, but
$P$ is necessarily finite when $\rep(P,\bbK)$ is a length category.
\end{rem}

\begin{rem}
  Consider a length category $\A=\rep(P,\bbK)$ as in
  Theorem~\ref{th:main}. The distributive object $M$ arising in
  condition (L2) is not necessarily unique. In fact, such objects are
  parameterised by $1$-cocycles $\mu=(\mu_{xy})$ in
  $H^*(\Sigma(P),\bbK^{\times})$, where $\Sigma(P)$ denotes the
  simplicial complex associated with $P$ and for $x< y$ the map
  $M(y)\to M(x)$ sends the basis element $m_y$ to $m_x\mu_{xy}$; see
  \cite{GS1983, IK2022}.
\end{rem}

\begin{rem}
  In Theorem~\ref{th:main} condition (L1) is needed.  Let
  $\A=\mod \bbK[\e]$ where $\bbK[\e]$ denotes the algebra of dual
  numbers over a field. This is not a category of poset
  representations, though conditions (L2) and (L3) hold. Take for $M$
  the indecomposable projective module.
\end{rem}
  
The following example is interesting for several reasons, and I am
grateful to Claus Michael Ringel for suggesting it.  The example shows
that in Theorem~\ref{th:main} condition (L3) is needed. Also, it yields
infinitely many objects $M$ of length two such that all objects are a
subquotient of a finite direct sum of copies of $M$.

\begin{exm}
Consider the matrix algebra
\[\La=\begin{bmatrix}
    \bbC&\bbC\oplus\bbC_\s\\ 0&\bbC
  \end{bmatrix}\] where the action of $\bbC$ on $\bbC_\s$ from the
right is twisted by conjugation; it is a finite dimensional
$\bbR$-algebra.  There are infinitely many isomorphism classes of
length two $\La$-modules, and all but two are faithful; see the Addendum
of \cite{DR1976}.
\end{exm}

\begin{rem}
  In Theorem~\ref{th:main} the assumption on the length category $\A$
  to be connected can be removed. Then one obtains a finite product
  $\A=\A_1\times\cdots\times\A_n$ of connected length categories and
  possibly different fields $Z(\A_1),\ldots,Z(\A_n)$.
\end{rem}  

\begin{rem}\label{re:sheaves}
  Any partially ordered set $P$ carries the structure of a topological
  space, by taking as open subsets the downward closed subsets of
  $P$. This assignment identifies finite posets with finite
  $T_0$-spaces \cite[\S1]{Al1937}.  Moreover, for any abelian category $\C$ the category
  of sheaves on $P$ with values in $\C$ identifies with the category
  of representations $P^\op\to\C$, by taking a sheaf $\mathscr F$ on
  $P$ to the representation
\[x\longmapsto \mathscr F_x=\mathscr F(\downarrow\! x)\qquad (x\in P)\]
where 
\[\downarrow\! x\coloneqq \{y\in P\mid y\le x\}.\]
Thus Theorem~\ref{th:main} characterises the categories of sheaves on
finite $T_0$-spaces with values in the category of finite dimensional
vector spaces over a field. One may compare this with Gabriel's
reconstruction of a noetherian scheme from its category of
quasi-coherent sheaves \cite[VI]{Ga1962}.
\end{rem}

\begin{rem}\label{re:birkhoff}
  We provide a formulation of Birkhoff's representation theorem for
  finite distributive lattices that is parallel to
  Theorem~\ref{th:birkhoff} for length categories.  Let $L$ be a
  finite distributive lattice and let $P\subseteq L$ denote the set
  of join-irreducible elements with the induced partial order. Write
  $\two=\{0<1\}$ for the poset having two elements. Then the
  assignment
\[x\longmapsto\card\Hom(-,x)|_{P}\]
induces an isomomorphism $L\iso\Hom(P^\op,\two)$.  The inverse map is given
by
\[\p\longmapsto\bigvee_{\substack{y\in P\\\p(y)=1}} y.\]
\end{rem}

\section{A generalisation}

In this section we consider infinite posets and its linear
representations over a field. This leads to a generalisation of
Theorem~\ref{th:main}. We need to introduce a class of abelian
categories that are controlled by its finite length objects but
contain objects of infinite length when there are infinitely many
simple objects.

\subsection*{Pointwise length categories}

We consider  abelian categories that admit local length
functions. For an object $X$ in an abelian category the
\emph{multiplicity} of a simple object $S$ is given by
\[ [X:S]\coloneqq \sup_{\substack{X'\subseteq X\\ X' \text{ finite
        length}}}[X':S].\] We say that $X$ is \emph{pointwise of
  finite length} if $X$ is the sum of its finite length subobjects and
the multiplicity $[X:S]$ is finite for every simple object $S$.  For
an exact sequence $0\to X'\to X\to X''\to 0$ of pointwise finite
length objects and any simple object $S$ one has
\[ [X':S]+ [X'':S]= [X:S].\]

\begin{defn}\label{de:ptw}
  A \emph{pointwise length category} is an abelian category that satisfies
  the following conditions:
\begin{enumerate}
\item Every object is pointwise of finite length.
\item The isomorphism classes of objects form a set.  
\item For every directed family of objects $(X_i)_{i\in I}$ the
  colimit exists provided that $\sup_{i\in I} [X_i:S]$ is finite for
  every simple object $S$.
\item For every subobject $U\subseteq X$ and directed family of
  subobjects $V_i\subseteq X$
  \[U\cap\left(\sum_i V_i\right)=\sum_i (U\cap V_i).\]  
\end{enumerate}
\end{defn}

For an abelian category $\A$ let $\fl\A$ denote its full subcategory
of finite length objects and $\ptfl\A$ the full subcategory of
pointwise finite length objects. The equality $\fl\A=\ptfl\A$ holds
precisely when the number of isomorphism classes of simple objects is
finite.

\begin{lem}
  Let $\A$ be a Grothendieck category. Then  $\ptfl\A$ is  a pointwise
  length category.
\end{lem}
\begin{proof}
  The only condition of the above definition that is not obvious is
  (3). Consider a directed family of objects $(X_i)_{i\in I}$ in
  $\ptfl\A$ and its colimit $X\coloneqq\colim_{i\in I}X_i$ in
  $\A$. This may be rewritten as a directed union $\sum_{i\in I}Y_i$
  of subobjects $Y_i\subseteq X$, by taking for $Y_i$ the image of
  $X_i\to X$. For a simple object $S$ one has
  \[[X:S]=\left[\sum_{i\in I}Y_i:S\right]=\sup_{i\in I} [Y_i:S]\le
    \sup_{i\in I} [X_i:S].\] Thus $X\in\ptfl\A$ when
  $\sup_{i\in I} [X_i:S]<\infty$ for every simple object $S$.
\end{proof}  
  
When $\A$ is essentially small we write $\Ind\A$ for its
ind-completion, which identifies with the category of left exact
functors $\A^\op\to\Ab$ via the assignment $X\mapsto\Hom(-,X)|_\A$;
see \cite[II]{Ga1962}. Note that $\Ind\A$ is a Grothendieck category.

\begin{prop}\label{pr:ind}
  Let $\A$ be a pointwise length category. Then the
  assignment \[X\longmapsto\Hom(-,X)|_{\fl\A}\] induces an equivalence
  \[\A\longiso  \ptfl(\Ind(\fl\A)).\]
\end{prop}
\begin{proof}
  Write any object in $\A$ as directed colimit of its finite length
  subobjects. For $X=\colim_{X'\subseteq X}X'$ and
  $Y=\colim_{Y'\subseteq Y}Y'$ we have
  \[\Hom(X,Y)\cong\lim_{X'\subseteq X}\Hom(X',Y)
    \cong\lim_{X'\subseteq X}\colim_{Y'\subseteq Y}\Hom(X',Y').\] From
  this it follows that $\A\to\Ind(\fl\A)$ is fully faithful, since it
  is fully faithful when restricted to $\fl\A$ by Yoneda's
  lemma. Every pointwise finite length object of $\Ind(\fl\A)$ is a
directed union of finite length subobjects and therefore in the essential image of the
  functor.
\end{proof}

\begin{cor}\label{co:ptfl}
A pointwise length category is determined by its
full subcategory of finite length objects. In particular, a functor
$\A\to\B$ between  pointwise length categories is an equivalence if it restricts
to an equivalence $\fl\A\iso\fl\B$ and preserves directed colimits.\qed
\end{cor}

From now on fix a pointwise length category $\A$ and choose a representative
set of simple objects $(S_x)_{x\in \Om}$ in $\A$.

We beging with a remark on Lemma~\ref{le:torsion}. The assertion of
this lemma extends to any
pointwise length category if we set for $I\subseteq \Om$ and $X\in\A$
\[\frt_I(X)\coloneqq\colim_{\substack{X'\subseteq X\\
      X'\in\fl\A}}\frt_I(X')\qquad\text{and}\qquad
  \frf_I(X)\coloneqq\colim_{\substack{X'\subseteq X\\
      X'\in\fl\A}}\frf_I(X')\] because taking directed colimits in
$\A$ is exact. In fact,
\[\A_I\coloneqq\{X\in\A\mid \frt_I(X)=X\}\]
is a Serre subcategory of $\A$; it is a pointwise length
category satisfying
\[\fl(\A_I)=(\fl\A)_I.\]

Now fix a commutative ring $\bbK$ and suppose that the category $\A$
is $\bbK$-linear. Moreover, we assume that $\fl\A$ has enough
projective objects and is \emph{hom-finite}, so $\Hom(X,Y)$ is a
finite length $\bbK$-module for all $X,Y\in\A$.  For each $x\in \Om$ fix
a projective cover $P_x\to S_x$.

Let $\proj\A$ denote the $\bbK$-linear category given by
$\Ob(\proj\A)\coloneqq \Om$ and
\[\Hom(x,y)\coloneqq \Hom(P_x,P_y)\] for $x,y\in \Om$.
For an object $M\in\A$ the assignment
$x\mapsto\Hom(P_x,M)$ induces a $\bbK$-linear functor
$\bar M\colon(\proj\A)^\op\to \mod\bbK$.

\begin{lem}\label{le:ptfl}
  The assignment $M\mapsto\bar M$ induces a 
  fully faithful functor
  \[h\colon\A\lto\Fun_\bbK((\proj\A)^\op,\mod\bbK)\] that preserves
  colimits. Moreover, it identifies $\fl\A$ with the subcategory of
  finite length objects of $\Fun_\bbK((\proj\A)^\op,\mod\bbK)$.
\end{lem}
\begin{proof}
  From the definition it is clear that $h$ preserves colimits.
  When $\Om$ is finite the assignment $F\mapsto \bigoplus_{x\in \Om}F(x)$
  identifies $\Fun_\bbK((\proj\A)^\op,\mod\bbK)$ with the category of finite
 length  modules over the endomorphism algebra of
  $P=\bigoplus_{x\in \Om}P_x$.  From this case it follows that
  \[\Hom(X,Y)\cong \Hom(\Hom(P,X),\Hom(P,Y))\cong\Hom(h(X),h(Y))\]
  when $X$ and $Y$ are of finite length, because we can assume that
  $\Om$ consists of the composition factors of $X$ and $Y$.  Now write
  any object in $\A$ as directed colimit of its finite length
  subobjects.  For $X=\colim_{X'\subseteq X}X'$ and
  $Y=\colim_{Y'\subseteq Y}Y'$ we compute
  \begin{align*}
    \Hom(X,Y)&    \cong\lim_{X'\subseteq X}\colim_{Y'\subseteq
               Y}\Hom(X',Y')\\
             &\cong\lim_{X'\subseteq X}\colim_{Y'\subseteq
               Y}\Hom(h(X'),h(Y'))\\
             &\cong\Hom(\colim_{X'\subseteq
               X}h(X'),\colim_{Y'\subseteq Y}h(Y'))\\
             &\cong\Hom(h(\colim_{X'\subseteq
               X}X'),h(\colim_{Y'\subseteq Y}Y'))\\
             &\cong\Hom(h(X),h(Y)).
  \end{align*}
The simple functors   $(\proj\A)^\op\to\mod\bbK$ are up to
isomorphism of the form
\[h(S_x)\cong\Hom(-,x)/\rad\Hom(-,x) \qquad (x\in \Om)\]
since $\End(x)$ is local. From this the last assertion follows.
\end{proof}

\subsection*{Representations of posets}

Let $P$ be a poset and $\bbK$ a field. We write $\bbK P$ for the
$\bbK$-linearisation of the category $P$ and
$\Fun_\bbK((\bbK P)^\op,\mod\bbK)$ denotes the category of
$\bbK$-linear functors $(\bbK P)^\op\to\mod\bbK$. Then restriction
along the inclusion $P\to\bbK P$ induces an equivalence
\[\Fun_\bbK((\bbK P)^\op,\mod\bbK)\iso\rep(P,\bbK).\]

Let $M_P$ denote the representation of $P$ given by $M_P(x)=\bbK$ and
$M_P(y)\to M_P(x)$ the identity map for all $x\le y$.

\begin{lem}\label{le:downfinite}
For $P$  and $\bbK$ the following are equivalent.
\begin{enumerate}
\item The poset $P$ is down-finite, that is,
  $\downarrow\! x$ is a finite set for each $x\in P$.
\item The category $\rep(P,\bbK)$ is a pointwise length category.
\item The representation $M_P$  is a sum of finite length subobjects. 
\end{enumerate}
\end{lem}
\begin{proof}
  (1) $\Rightarrow$ (2) All one needs to show is that every object in
  $\rep(P,\bbK)$ is a sum of finite length subobjects, because the
  other conditions in Definition~\ref{de:ptw} are automatic.

  Let $M\in \rep(P,\bbK)$. For $x\in P$ consider
  the subrepresentation $M_x\subseteq M$
\[M_x(y)\coloneqq \begin{cases} M(y)&\text{if }y\le x,\\
    0& \text{otherwise}.
\end{cases}\]
Then $M_x$ has finite length because  $P$ is down-finite, and
$M=\sum_{x\in P}M_x$.

(2) $\Rightarrow$ (3) Clear.
 
 (3) $\Rightarrow$ (1) Let $x\in P$ and $N\subseteq M_P$ a
 subrepresentation such that $N(x)\neq 0$. Then $N(y)\neq 0$ for all
 $y\le x$.  Thus if $M_P$ is a sum of finite length subobjects, then
 $P$ needs to be down-finite.
\end{proof}

The above lemma provides the obstruction on the class of posets we can
deal with when generalising Theorem~\ref{th:main} to arbitrary posets.

\begin{thm}\label{th:main2}
  A connected pointwise length category is equivalent to the category
  of pointwise finite dimensional representations of a poset over a
  field if and only if the following hold:
  \begin{enumerate}
  \item[(L1)] The ext-quiver admits only finitely many paths ending in a fixed vertex. 
  \item[(L2)] There is an object $M$ with a distributive lattice of
    subobjects such that each  finite length object is a subquotient of a finite
    direct sum of copies of $M$.
  \item[(L3)] For each simple object all its endomorphisms are central. 
\end{enumerate}
In this case the field equals the centre of the category and the poset is
given by the isomorphism classes of simple objects, with $[S]\le [T]$
if there is a path from $[S]$ to $[T]$ in the ext-quiver.
\end{thm}

\begin{proof}
  Fix a connected pointwise length category $\A$.  Suppose first that
  $\A\iso\rep(P,\bbK)$ for a poset $P$ and a field $\bbK$. Then $P$ is
  down-finite by Lemma~\ref{le:downfinite}. We need to show that the
  conditions (L1)--(L3) hold. The proof of Theorem~\ref{th:main}
  carries over. For instance, (L1) holds because the ext-quiver
  identifies with the Hasse diagram of $P$, and (L3) is clear as well.
  For (L2) we can reduce to the case that $P$ is finite, because any
  finite subset of $P$ is contained in a finite downward closed subset
  $P'\subseteq P$. Then $\A_{P'}$ is a length category and
  $\A_{P'}\iso\rep(P',\bbK)$. Moreover one uses that
  $\frt_{P'}(M_P)=M_{P'}$. Thus the object
  $M_P=\sum_{P'\subseteq P} M_{P'}$ is distributive since all
  finite length subobjects are distributive and finite intersections
  distribute over directed unions. It follows that (L2) holds for $\A$,
  given that
  \[\fl\A=\bigcup_{P'\subseteq P}\A_{P'}.\]

  Now suppose that conditions (L1)--(L3) hold for $\A$. First observe
  that \[\bbK\coloneqq Z(\fl\A)\cong Z(\A)\] is a field by
  Lemma~\ref{le:centre}. Choose a representative set of simple
  objects $(S_x)_{x\in \Om}$ in $\A$. Then $\Om$ is partially ordered  via $x\le y$ if
  there is a path from $S_x$ to $S_y$ in the ext-quiver of $\A$. 

  For $x\in \Om$ consider the subobject
  \[M_x\coloneqq \frt_{\downarrow x}(M)\subseteq M.\] We claim that $M_x$
  has finite length and that it is a projective cover of $S_x$. We
  apply Lemma~\ref{le:torsion} using the decomposition
  $\Om=\downarrow\!  x\sqcup(\Om\smallsetminus\downarrow\! x)$. Note that
  $\downarrow\!  x$ is a finite set because of (L1).
Thus $\A_{\downarrow x}$ is a length category and we wish to apply
  Theorem~\ref{th:main}.  It is easily checked that (L1)--(L3) hold; in
  particular $M_x$ is distributive and all objects of
  $\A_{\downarrow x}$ are subquotients of finite direct sums of $M_x$
  since the functor $\frt_{\downarrow x}$ is exact. Thus $M_x$ is
  projective in $\A_{\downarrow x}$ by Lemma~\ref{le:max},
  and then also projective in $\A$ since the inclusion
  $\A_{\downarrow x}\to\A$ is left adjoint to an exact functor.

  We have $\Hom(M_x,M_y)\cong\bbK$ if and only if $x\le y$, and a
  basis element is given by the inclusion $M_x\to M_y$. This follows
  as in the proof of Theorem~\ref{th:main} and yields a $\bbK$-linear
  equivalence $\bbK \Om\iso\proj\A$, where $\bbK \Om$ denotes the
  $\bbK$-linearisation of $\Om$.  Now consider the composite
\[\A\xto{\ h\ }\Fun_\bbK((\proj\A)^\op,\mod\bbK)\iso\Fun_\bbK((\bbK
  \Om)^\op,\mod\bbK) \iso\rep(\Om,\bbK)\]
where the first  functor is
from Lemma~\ref{le:ptfl}. This is fully faithful and we need to
show that it is essentially surjective. The functor sends $M$ to $M_\Om$
and this is a sum of finite length objects. Thus $\rep(\Om,\bbK)$ is a
pointwise length category by Lemma~\ref{le:downfinite}. The functor
identifies $\fl\A$ with the full subcategory of finite length objects
of $\rep(\Om,\bbK)$ by Lemma~\ref{le:ptfl}. Thus it follows from
Corollary~\ref{co:ptfl} that $\A\to \rep(\Om,\bbK)$ is an equivalence.
\end{proof}

\begin{rem}
  For a general poset $P$ and a field $\bbK$ there is no chance to
  reconstruct $P$ from the finite length objects in $\rep(P,\bbK)$.
  For example, the subcategory of finite length objects is semisimple
  when $P=(\bbQ,\le)$. More precisely, $\Ext^1(S_x,S_y)\neq 0$ if and
  only if $y<x$ and there is no $z$ with $y<z<x$.
\end{rem}

\subsection*{Acknowledgements}

I wish to thank Gustavo Jasso for his comments and suggestions
concerning a version of the main theorem for infinite posets. Thanks
also to Dave Benson for suggesting the term `pointwise length
category' and to Maximilian Kaipel for various helpful comments on an
earlier version of this note. The author is grateful to the Max Planck
Institute for Mathematics in Bonn for its hospitality and financial
support. This work was supported by the Deutsche
Forschungsgemeinschaft (SFB-TRR 358/1 2023 - 491392403).

\end{document}